\documentclass[12pt]{amsart}

\usepackage{amssymb}

\usepackage{enumitem}

\usepackage{amsmath,mathrsfs,cancel}

\usepackage{hyperref}
\usepackage{tikz}
\usepackage{color}

\makeatletter
\@namedef{subjclassname@2020}{%
  \textup{2020} Mathematics Subject Classification}
\makeatother

\usepackage[T1]{fontenc}

\newtheorem{theorem}{Theorem}[section]

\newtheorem{proposition}[theorem]{Proposition}

\theoremstyle{definition}
\newtheorem{definition}[theorem]{Definition}
\newtheorem{remark}[theorem]{Remark}

\numberwithin{equation}{section}

\def\al{\alpha}
\def\nn{\nonumber}

\def\nn{\nonumber}

\def\Z{{\mathbb Z}}
\def\C{{\mathbb C}}

\def\R{\mathbb R}

\def\rddots{\displaystyle\cdot^{\displaystyle\cdot^{\displaystyle\cdot}}}
\def\wlong{w_{\mathrm{long}}}

\newcommand\smallf[2]{{\textstyle{\frac{#1}{#2}}}}

\DeclareMathOperator{\GL}{GL}
\DeclareMathOperator{\SL}{SL}
\DeclareMathOperator{\re}{Re}

\begin{document}


\baselineskip=17pt


\title[First coefficient of Langlands Eisenstein series]{\boldmath The first coefficient of Langlands Eisenstein series for $\text{\rm \bf SL}(n,\mathbb Z)$}

\author[D. Goldfeld ]{Dorian Goldfeld}
\address{Department of Mathematics\\Columbia University\\ 2990 Broadway \\ New York, NY 10027, USA}
\email{goldfeld@columbia.edu}

\author[E. Stade]{Eric Stade}
\address{Department of Mathematics\\
University of Colorado Boulder\\
Boulder, Colorado 80309, USA}
\email{stade@colorado.edu}

\author[M. Woodbury]{Michael Woodbury}
\address{Department of Mathematics \\Rutgers, The State University of New Jersey\\
110 Frelinghuysen Rd\\
Piscataway, NJ 08854-8019, USA}
\email{michael.woodbury@rutgers.edu}

\date{}
\dedicatory{Dedicated to Henryk Iwaniec on the occasion of his 75th birthday}

\begin{abstract}
Fourier coefficients of Eisenstein series figure prominently in the study of automorphic L-functions via the Langlands-Shahidi method, and in various other aspects of the theory of automorphic forms and representations.
 
In this paper, we define Langlands Eisenstein series for ${\rm SL}(n,\mathbb Z)$ in an elementary manner, and then determine the first Fourier coefficient of these series in a very explicit form.  Our proofs and derivations  are short and simple, and use the Borel Eisenstein series as a template to determine the first Fourier coefficient of other Langlands Eisenstein series.
\end{abstract}

\subjclass[2020]{Primary 11F55; Secondary11F72}

\keywords{Langlands Eisenstein series, Fourier-Whittaker  coefficients, automorphic forms, automorphic L functions}

\maketitle

\section{Introduction}

   The classical upper half-plane is the set of all complex numbers $$\mathfrak h^2:=\{x+iy 
   \mid x\in\mathbb R,\, y>0\},$$
   which can also be realized, in group theoretic terms, by the Iwasawa decomposition (see \cite{Goldfeld2006}) as
   \begin{align*}
   \mathfrak h^2 = \GL(2, \mathbb R)/(\text{\rm O}(2,\mathbb R)\cdot \mathbb R^\times) = \left\{\begin{pmatrix} y&x\\0&1\end{pmatrix} \; \Big | \; x\in\mathbb R, \,y>0\right\}.
   \end{align*}
   For $g = \left(\begin{smallmatrix} y&x\\0&1 \end{smallmatrix}\right) \in\mathfrak h^2$ and $s\in\mathbb C$ we define the power function
   \begin{equation}\label{PowerFunction-y^s}
   I_s(g) := y^s.
   \end{equation}
   Let $\Gamma_\infty = \big\{ \left(\begin{smallmatrix} 1&m\\0&1 \end{smallmatrix}\right) \, \big| \, m\in\mathbb Z\big\}.$
  Then the non-holomorphic Eisenstein series for $\text{\rm SL}(2,\mathbb Z)$ is defined for $\text{\rm Re}(s) > 1$ by the convergent series
   $$\text{\rm E}(g, s) := \sum_{\gamma\in \Gamma_\infty\backslash \text{\rm SL}(2,\mathbb Z)}  \tfrac12\cdot I_s(\gamma g).$$
   
  For $g = \left(\begin{smallmatrix} y&x\\0&1 \end{smallmatrix}\right) \in\mathfrak h^2$ (with $y$ fixed), the Eisenstein series $\text{\rm E}(g, s)$ has a Fourier expansion in the $x$-variable given by
  \begin{align} \label{FourierExpEisensteinGL(2)}
 &E(g,s) = \underset{\text{\bf constant term}}{\underbrace{y^s + \phi(s) y^{1-s}}} \\ \nn&\hskip50pt + \hskip-2pt
\underset{  \text{\bf first Fourier coeff.}} {\underbrace{\frac{2}{\pi^{-s}\Gamma(s)\zeta(2s)}}}
\sum_{n\ne0} \;\, \underset{\text{\bf Hecke eigenvalue}}{\underbrace{ \sigma_{1-2s}(n)  |n|^{s-\frac12}}}\;   {W_{2,s-\frac{1}{2}}}(|n|y)\cdot e^{2\pi inx}
\end{align}
where
 $$\phi(s) = \sqrt{\pi}\,\frac{\Gamma(s-\tfrac12)\zeta(2s-1)}{\Gamma(s)\zeta(2s)}, \quad\sigma_s(n) = \sum_{\substack{d|n\\ d>0}} d^s,$$and $$W_{2,\al}(y) = \frac{\sqrt{y}}{2}\int\limits_0^\infty e^{- \pi y (u+{1/ u})}\,u^{\al}\,\frac{du}{u}.$$

 The Fourier expansion (\ref{FourierExpEisensteinGL(2)}) is one of the most important in the theory of  modular forms. We have singled out the \emph{``constant term,''} the \emph{``first Fourier coefficient,''} and the \emph{``Hecke eigenvalue,''} which have each played a significant role in the history of the subject. 
 
Let $F$ be a number field with associated adele ring $\mathbb A_F$. The constant term of the Fourier expansion of Langlands Eisenstein series for a quasi-split group over $\mathbb A_F$ has been known for a long time (see \cite{Langlands1976}, \cite{MR0419366}, \cite{MR951897}). The Langlands-Shahidi method (first introduced in \cite{MR610479}) is a method to compute local coefficients for generic representations of reductive groups. In the case of Eisenstein series, Shahidi uses the Casselman-Shalika formula for Whittaker functions to express the the first coefficient as a product of L-functions (see \cite{MR816396}, \cite{MR1070599}). This gives a new proof of the analytic continuation and functional equation of Rankin-Selberg L-functions since they occur in the non-constant term of certain Eisenstein series.

The Langlands-Shahidi method of studying L-functions by way of Eisenstein series has numerous applications.  For example, Kim and Shahidi  apply this method to the analysis of $\GL(2)\times\GL(3)$ tensor product  representations \cite{MR1923967}, and to the symmetric cube representation on $\GL(2)$   \cite{MR1726704} \cite{MR1923967}, deriving functoriality results in both cases.  Further, from the symmetric cube result, they are able to advance the state of the art concerning the Ramanujan-Petersson and Selberg conjectures for $\GL(2)$, obtaining an upper bound of $5/34$ for Hecke eigenvalues of $\GL(2)$ Maass forms, over any number field and at any prime (finite or infinite).
 
In additional work, Kim \cite{MR1937203} uses the Langlands-Shahidi method to obtain functoriality results concerning exterior  square representations on $\GL(4)$, and symmetric fourth power representations on $\GL(2)$.   As a consequence of the latter result, Kim and Sarnak \cite[Appendix 2]{MR1937203} obtain a lower bound $\lambda_1\ge 975/4096\approx 0.238$ for the first eigenvalue of the Laplacian, acting on the corresponding hyperbolic space.  Moreover, in \cite{MR1890650}, Kim and Shahidi prove a criterion for cuspidality of the $\GL(2)$ symmetric fourth power representation, and deduce from this a number of results towards the Ramanujan-Petersson and Sato-Tate conjectures.
 
In further work, Kim \cite{MR2380325} applies the Langlands-Shahidi method to  exceptional groups.  In this context, various other types of L-functions arise, and a number of results concerning the holomorphy of these L-functions follow.   
 
There are numerous other applications and potential applications, some of which are discussed in \cite{MR1937203}.  In sum, information concerning Fourier coefficients of Eisenstein series is central to the Langlands-Shahidi method, which has proved a powerful tool  in the theory of automorphic forms and representations, and has strong potential for  relevance to additional  Langlands functoriality and related results.

The main goal of this paper is to first define Langlands Eisenstein series for  $\text{\rm SL}(n,\mathbb Z)$ in an elementary manner, and then determine the first Fourier coefficient of the Langlands Eisenstein series  in a very explicit form. This result is stated in Theorem \ref{Thrm:FirstEisCoeff} which is the main theorem of this paper. The proof of this theorem is also short and simple (following the methods introduced in \cite{GMW2021}) using the Borel Eisenstein series as a template to determine the first Fourier coefficient of other Eisenstein series.

\section{Basic functions on the generalized upper half plane $\mathfrak h^n$}

For an integer $n\ge 2$, let $U_n(\R)\subseteq \GL(n,\R)$ denote the group of upper triangular unipotent matrices and let
  $\text{O}(n,\R)\subseteq \GL(n,\R)$ denote the group of real orthogonal matrices.  

\begin{definition}[Generalized upper half plane]
  We define the \emph{generalized upper half plane} as 
 \[ \mathfrak h^n := \GL(n,\mathbb R)/\left(\text{O}(n,\R)\cdot\R^\times   \right). \]
 By  the Iwasawa decomposition of $\GL(n)$ (see \cite{Goldfeld2006}) every element of $\mathfrak h^n$ has a coset representative of the form $g=x y$ where
\begin{equation}\label{eq:ymatrix-def}
x = \left(\begin{smallmatrix} 
  1 & x_{1,2} & x_{1,3}& \cdots  & & x_{1,n}\\
  & 1& x_{2,3} &\cdots & & x_{2,n}\\
  & &\hskip 2pt \ddots & & & \vdots\\
  & && & 1& x_{n-1,n}\\
  & & & & &1\end{smallmatrix}\right) \in {U}_n(\mathbb R), \qquad\;
 y =
    \left(\begin{smallmatrix} y_1y_2\cdots
    y_{n-1} & & & \\
    & \hskip -30pt y_1y_2\cdots y_{n-2} & & \\
    & \ddots &  & \\
    & & \hskip -5pt y_1 &\\
    & & &  1\end{smallmatrix}\right),
\end{equation}
with $y_i > 0$ for each $1 \le i \le n-1$.  The group $\GL(n,\R)$ acts as a group of transformations on $\mathfrak h^n$ by left multiplication.
\end{definition}

\begin{definition}[{Character of $U_n(\R)$}]
Let $M=(m_1,\ldots,m_{n-1})\in \Z^{n-1}$.  For an element $x\in U_n(\R)$ of the form
\begin{equation}\label{eq:xmatrix-def}
 x = \left(\begin{smallmatrix} 
  1 & x_{1,2} & x_{1,3}& \cdots  & & x_{1,n}\\
  & 1& x_{2,3} &\cdots & & x_{2,n}\\
  & &\hskip 2pt \ddots & & & \vdots\\
  & && & 1& x_{n-1,n}\\
  & & & & &1\end{smallmatrix}\right),
\end{equation}
we define {\it the character $ \psi_M$} by
\begin{equation}\label{eq:psiMdef}
 \psi_M(x):= m_1x_{1,2}+m_2x_{2,3}+\cdots+m_{n-1} x_{n-1,n}.
\end{equation}
\end{definition}

Next, we generalize  the power function (\ref{PowerFunction-y^s}) which is used to construct the Eisenstein series for $\SL(2,\mathbb Z).$
\begin{definition}[{Power function}] Fix an integer $n\ge 2.$ 
Let $$\alpha=(\alpha_1,\ldots,\alpha_n)\in \C^n$$ with $\alpha_1+\cdots +\alpha_n=0$.  Let $\rho=(\rho_1,\ldots,\rho_n)$, where $\rho_i=\frac{n+1}{2}-i$ for $i=1,2,\ldots,n$.   We define a \emph{power function} on $xy\in \mathfrak h^n$  by
\begin{equation}\label{Ialpha}
I_n(xy,\alpha) =\prod_{i=1}^n d_i^{\alpha_i+\rho_i}= \prod_{i=1}^{n-1}y_i^{\alpha_1+\cdots+\alpha_{n-i}+\rho_1+\cdots+\rho_{n-i}},
\end{equation}
where $d_i=\prod\limits_{j\le n-i}y_j$ is the $j$-th diagonal entry of the matrix $g=xy$ as above.
\end{definition}

\begin{definition}[{Weyl group}] \label{WeylGroup}
Let $W_n \cong S_n$ denote the \emph{Weyl group} of $\GL(n, \mathbb R).$  We consider it as the subgroup of $\GL(n,\R)$ consisting of permutation matrices, i.e., matrices that have exactly one $1$ in each row/column and all zeros otherwise.   The \emph{long element of $W_n$} is $\wlong:=\left( \begin{smallmatrix} & & 1 \\ & \rddots & \\ 1 & & \end{smallmatrix}\right)$.
\end{definition}

\begin{definition}[{Jacquet's Whittaker function}]\label{def:JacWhittFunction}
Let $g\in \GL(n,\R)$ with $n\geq 2$.  Let $\alpha = (\alpha_1,\alpha_2, \;\ldots, \;\alpha_n)\in \C^n$ with $\alpha_1+\cdots +\alpha_n=0$.  We define the \emph{completed    Whittaker function} $W^{\pm}_{n,\alpha}: \GL(n,\mathbb R)\big/\left(
  \text{O}(n,\mathbb R)\cdot \mathbb R^\times\right) \to \mathbb C$   by the integral
$$W^{\pm}_{n,\alpha}(g) := \prod_{1\leq j< k \leq n} \frac{\Gamma\big(\frac{1 + \alpha_j - \alpha_k}{2}\big)}{\pi^{\frac{1+\alpha_j - \alpha_k}{2}} }\cdot \int\limits_{U_4(\mathbb R)} I_n(w_{\mathrm{long}} ug,\alpha) \,\overline{\psi_{1, \ldots, 1,\pm 1}(u)} \, du, $$
which converges absolutely if $\re(\al_i-\al_{i+1})>0$ for $1\le i\le n-1$ (cf.  \cite{GMW2021}), and has meromorphic continuation to all $\alpha\in \C^n$ satisfying $\alpha_1+\cdots +\alpha_n=0$.
\end{definition}

\begin{remark}
With the additional Gamma factors included in this definition (which can be considered as a \emph{``completed''} Whittaker function) there are $n!$ functional equations. This is equivalent to the fact that the Whittaker function is invariant under all permutations of  $\alpha_1,\alpha_2, \ldots, \alpha_n$.  Moreover, even though the integral (without the normalizing factor) often vanishes identically as a function of $\al$, this normalization never does.

If $g$ is a diagonal matrix in $\GL(n,\mathbb R)$ then the value of $W^{\pm}_{n,\alpha}(g)$ is independent of sign, so we drop the $\pm$. We also drop the $\pm$ if the sign is $+1$.
\end{remark}

\section{The Borel Eisenstein series for $\SL(n,\mathbb Z)$}

The {Borel  subgroup} $\mathcal B$ for $\GL(n,\mathbb R)$ is given by
$$\mathcal B=\left\{\left(\begin{smallmatrix}
*&* &\cdots &*\\
& * &\cdots & *\\
&&\ddots &\vdots \\
& & & *\end{smallmatrix}\right) \subset \GL(n,\mathbb R)   \right\}.$$
Among the general parabolic subgroups  defined in Definition \ref{GLnParabolic}, the Borel  subgroup is  minimal. The Borel Eisenstein series $E_\mathcal B(g, {\alpha})$ for $\SL(n,\mathbb Z)$ is a complex-valued function of  variables $g\in \GL(n,\mathbb R)$ and $\alpha=(\alpha_1,\ldots,\alpha_n)\in \mathbb C^n$ where $\alpha_1+\cdots+\alpha_n=0.$ For $\Gamma_n:=\SL(n,\mathbb Z)$ and $\text{\rm Re}(\alpha_i)-\text{\rm Re}(\alpha_{i+1})> 1,$ {\small$(i=1,\ldots,n-1)$},  it is defined by the absolutely convergent series
\begin{equation}
E_\mathcal B(g,\alpha) := \sum_{(\mathcal B\,\cap\,\Gamma_n)\backslash \Gamma_n} I_n(\gamma g, \alpha).\end{equation}
\begin{proposition}[The $M^{\rm{th}}$ Fourier-Whittaker coefficient of  $E_{\mathcal B}$] \label{LanglandsParPMin}  Define the vector $M := (m_1,m_2,\ldots, m_{n-1}) \in \mathbb Z_{+}^{n-1}$ and the matrix
$$M^*:=  \left(\begin{smallmatrix} m_1m_2\cdots
    m_{n-1} & & & \\
     \qquad\quad \ddots&  & & \\
    &\hskip-8pt m_1m_2 &  & \\
    & & \hskip -12pt m_1 &\\
    & & &  1\end{smallmatrix}\right).$$ Then the $M^{th}$ term in the  Fourier-Whittaker expansion of $E_{\mathcal B}$ (see \cite{Goldfeld2006}) is given by
\begin{align*}
\int\limits_{U_n(\mathbb Z)\backslash U_n(\mathbb R)} E_{\mathcal B}(ug, \alpha)\, \overline{\psi_M(u) }\; du \; = \; \frac{A_{E_{\mathcal B}}(M,\alpha)}
{\prod\limits_{k=1}^{n-1}  m_k^{k(n-k)/2}} \; W_{n,\alpha}\big(M^* g\big),
\end{align*}
where $A_{E_{B}}(M,\alpha) = A_{E_{B}}\big((1,\ldots,1),\alpha\big) \cdot \lambda_{E_{B}}(M,\alpha),$
and
\begin{equation}\label{HeckeEigenvaluePMin}
\lambda_{E_{B}}\big((m,1,\ldots, 1), \alpha\big) = \underset{c_1c_2\cdots c_n=m} {\sum_{c_1,\ldots,c_n\in\mathbb Z_{+}}} c_1^{\alpha_1} c_2^{\alpha_2}\cdots  c_n^{\alpha_n}, \qquad (m\in\mathbb Z_{+})
\end{equation}
  is the $(m,1,\ldots,1)^{th}$ (or more informally the $m^{th}$) Hecke eigenvalue of $E_{B}$.
  \end{proposition}

  \begin{proof} See \cite{Goldfeld2006}.
  \end{proof}

 \begin{proposition}[The first Fourier coefficient of  $E_{\mathcal B}$] \label{FirstCoeffPMin} We have
$$
A_{E_{B}}\big((1,\ldots,1), \alpha\big) = \, c_0 \prod_{1\leq j< k \leq n} \zeta^*\big(1 + \alpha_j - \alpha_k \big)^{-1}
$$
for some constant $c_0 \ne 0$ (depending only on $n$),   and  $$\zeta^*(w) = \pi^{-\frac{w}{2}}\Gamma\left( \frac{w}{2} \right)\zeta(w)$$ is the completed Riemann $\zeta$-function.
\end{proposition}

\begin{proof} See \cite{GMW2021}.
\end{proof}

\section{Eisenstein series attached to lower-rank Maass cusp forms on Levi components}

\begin{definition}[{Langlands parameters}]
Let $n\geq 2$.  A vector $$\alpha=(\alpha_1,\ldots,\alpha_n)\in \C^n$$ is termed a \emph{Langlands parameter} if $\alpha_1+\cdots+\alpha_n=0$.
\end{definition}

\begin{definition}[{Maass cusp forms}]
 Fix $n\ge 2.$ A \emph{Maass cusp form} with Langlands parameter $\alpha\in \C^n$ for $\SL(n,\mathbb Z)$ is a smooth function $\phi: \mathfrak{h}^n \to \mathbb C$ which satisfies
 $\phi(\gamma g) = \phi(g)$
 for all $\gamma\in \SL(n,\Z)$, $g\in \mathfrak{h}^n$.  In addition, $\phi$ is square integrable and has the same eigenvalues under the action of the algebra of $\GL(n,\R)$ invariant differential operators on $\mathfrak{h}^n$ as the power function $I_n(*,\alpha)$. The Laplace eigenvalue of $\phi$ is given by (see  Section~6 in \cite{Miller_2002})
   $$\frac{n^3 - n}{24} -\frac{\alpha_1^2+\alpha_2^2+\cdots+\alpha_n^2}{2}.$$
The Maass cusp form $\phi$ is said to be {\it tempered at infinity} if the coordinates  $\alpha_1,\ldots, \alpha_n$ of the Langlands parameter are all pure imaginary.
\end{definition}

\begin{definition}[Parabolic subgroups] \label{GLnParabolic} For $n\ge 2$ and $1\le r\le n,$ consider a partition of $n$ given by
$n = n_1+\cdots +n_r$ with positive integers $n_1,\cdots,n_r.$  We define the \emph{standard parabolic subgroup} $$\mathcal P := \mathcal P_{n_1,n_2, \ldots,n_r} := \left\{\left(\begin{matrix} \GL(n_1) & * & \cdots &*\\
0 & \GL(n_2) & \cdots & *\\
\vdots & \vdots & \ddots & \vdots \\
0 & 0 &\cdots & \GL(n_r)\end{matrix}\right)\right\}.$$

Letting $I_r$ denote the $r\times r$ identity matrix, the subgroup
$$N^{\mathcal P} := \left\{\left(\begin{matrix} I_{n_1} & * & \cdots &*\\
0 & I_{n_2} & \cdots & *\\
\vdots & \vdots & \ddots & \vdots \\
0 & 0 &\cdots & I_{n_r}\end{matrix}\right)\right\}$$
is the unipotent radical of $\mathcal P$.  The subgroup
$$M^{\mathcal P} := \left\{\left(\begin{matrix} \GL(n_1) & 0 & \cdots &0\\
0 & \GL(n_2) & \cdots & 0\\
\vdots & \vdots & \ddots & \vdots \\
0 & 0 &\cdots & \GL(n_r)\end{matrix}\right)\right\}$$
 is the {\it Levi subgroup} of $\mathcal P$.
\end{definition}

 \begin{definition}[Maass form $\Phi$ associated to a parabolic $\mathcal P$] \label{def:maassparabolic}\label{InducedCuspForm} Let $n\ge 2$. Consider a partition $n = n_1+\cdots +n_r$ with $1 < r < n$. Let  $\mathcal P := \mathcal P_{n_1,n_2, \ldots,n_r} \subset \GL(n,\mathbb R).$
  For $i = 1,2,\ldots, r$, let
$\phi_i:\GL(n_i,\mathbb R)\to\mathbb C$ be either the constant function 1 (if $n_i=1$) or a Maass cusp form for $\SL(n_i,\mathbb Z)$ (if $n_i>1$).  The \emph{Maass form} $\Phi := \phi_1\otimes \cdots\otimes \phi_r$ is defined on $\GL(n,\mathbb R)=\mathcal P(\mathbb R)K$ (where $K= \text{\rm O}(n,\mathbb R))$ by the formula
$$\Phi(n m k ) := \prod_{i=1}^r \phi_i(m_i), \qquad (n\in  N^{\mathcal P}, m\in  M^{\mathcal P},k\in K)$$
where
  $m \in M^{\mathcal P}$ has the form
$m = \left(\begin{smallmatrix} m_1 & 0 & \cdots &0\\
0 & m_2 & \cdots & 0\\
\vdots & \vdots & \ddots & \vdots \\
0 & 0 &\cdots & m_r\end{smallmatrix}\right)$, with    $m_i\in \GL({n_i},\mathbb R).$  In fact, this construction works equally well if some or all of the $\phi_i$ are Eisenstein series.

 \end{definition}

\begin{definition}[Character of a parabolic subgroup]\label{ParabolicNorm}  Let $n\ge 2.$ Fix a partition
$n = n_1+n_2 + \cdots +n_r$ with associated parabolic subgroup $\mathcal P = \mathcal P_{n_1,n_2, \ldots,n_r}.$  Define
\begin{equation}\label{rhoPj}
\rho_{_{\mathcal P}}(j)=\left\{
                            \begin{array}{ll}
                              \frac{n-n_1}{2}, & j=1 \\
                              \frac{n-n_j}{2}-n_1-\cdots-n_{j-1}, & j\ge 2.
                            \end{array}
                          \right.
\end{equation}
 Let $s = (s_1,s_2, \ldots, s_r) \in\mathbb C^r$ satisfy
$\sum\limits_{i=1}^r n_i s_i = 0.$
Consider the function
$$| \cdot   |_{_{\mathcal P}}^s := I(\cdot,\alpha)$$
 on $\GL(n,\mathbb R)$, where
\begin{multline*}
\alpha=(\overbrace{s_1-\rho_{_{\mathcal P}}(1)+\smallf{1-n_1}{2}, \; s_1-\rho_{_{\mathcal P}}(1)+\smallf{3-n_1}{2}, \; \ldots \; ,s_1-\rho_{_{\mathcal P}}(1)+\smallf{n_1-1}{2}}^{n_1 \;\,\text{\rm terms}},\\ \quad \ 
\overbrace{s_2-\rho_{_{\mathcal P}}(2)+\smallf{1-n_2}{2}, \; s_2-\rho_{_{\mathcal P}}(2)+\smallf{3-n_2}{2}, \; \ldots \; ,s_2-\rho_{_{\mathcal P}}(2)+\smallf{n_2-1}{2}}^{n_2 \;\,\text{\rm terms}},\\
\vdots
\\\  \  \ 
\ldots \;\; ,\overbrace{s_r-\rho_{_{\mathcal P}}(r)+\smallf{1-n_r}{2}, \; s_r-\rho_{_{\mathcal P}}(r)+\smallf{3-n_r}{2}, \; \ldots \; ,s_r-\rho_{_{\mathcal P}}(r)+\smallf{n_r-1}{2}}^{n_r \;\,\text{\rm terms}}).
\end{multline*}
The conditions $\sum\limits_{i=1}^r n_i s_i = 0$ and $\sum\limits_{i=1}^r n_i \rho_{_{\mathcal P}}(i) = 0$
 guarantee that the entries of  $\alpha$  sum to zero.  When $g\in \mathcal P$, with diagonal block entries $m_i\in \GL(n_i,\mathbb R)$, one has
 \vskip-10pt
 $$| g  |_{_{\mathcal P}}^s=\prod_{i=1}^r\left| \text{\rm det}(m_i)\right|^{s_i},$$
   so that $| \cdot   |_{_{\mathcal P}}^s$ restricts to a character of $\mathcal P$ which is trivial on $N^{\mathcal P}$.   
 \end{definition}
 
  \begin{definition}[Langlands Eisenstein series attached to  Maass cusp forms of lower rank] \label{EisensteinSeries}
Let $\Gamma = \SL(n, \mathbb Z)$ with $n \ge 2.$  Consider a parabolic subgroup $\mathcal P = \mathcal P_{n_1,\ldots,n_r}$ of $\GL(n,\mathbb R)$ and functions $\Phi$  and $| \cdot |_{_{\mathcal P}}^s$   as given in Definitions \ref{InducedCuspForm} and  \ref{ParabolicNorm}, respectively. Let $$s = (s_1,s_2, \ldots,s_r)\in\mathbb C^r, \;\;\text{where}\;\; \sum_{i=1}^r n_is_i =0.$$
The \emph{Langlands Eisenstein series} determined by this data is defined by
\begin{equation}\label{def:EPhi}
 E_{\mathcal P, \Phi}(g,s) := \sum_{\gamma\,\in\, (\mathcal P\, \cap \,\Gamma)\backslash \Gamma}  \Phi(\gamma g)\cdot |\gamma g|^{s+\rho_{_{\mathcal P}}}_{_{\mathcal P}}
 \end{equation}
as an absolutely convergent sum for $\re(s_i)$ sufficiently large, and extends to all $s\in \C^r$ by meromorphic continuation.
\end{definition}

For $k=1,2,\ldots,r,$ let $\alpha^{(k)} := (\alpha_{k,1},\ldots,\alpha_{k,n_k})$ denote the Langlands parameters of $\phi_k.$ We adopt the convention that if $n_k=1$ then $\alpha_{k,1} = 0.$
Then the Langlands parameters of $E_{\mathcal P, \Phi}(g,s)$ (denoted $\alpha_{_{\mathcal P,\Phi}}(s)$) are
\begin{align}\label{langlandsparamsforPhi}
&\bigg (\overbrace{\alpha_{1,1}+s_1, \;\ldots \;,\alpha_{1,n_1}+s_1}^{n_1 \;\,\text{\rm terms}}, \quad\overbrace{\alpha_{2,1}+s_2, \;\ldots \;,\alpha_{2,n_2}+s_2}^{n_2 \;\,\text{\rm terms}},\\\nn&\hskip160pt
\quad\ldots \quad
,\overbrace{\alpha_{r,1}+s_r, \;\ldots \;,\alpha_{r,n_r}+s_r}^{n_r \;\,\text{\rm terms}}\bigg). 
\end{align}

\begin{proposition}[The $M^{\rm{th}}$ Fourier coefficient of   $E_{\mathcal P, \Phi}$] \label{MthEisCoeff}
 Let $$s = (s_1,s_2, \ldots,s_r)\in\mathbb C^r,$$  where $\sum\limits_{i=1}^r n_is_i =0.$
Consider  $E_{\mathcal P, \Phi}(*,s)$ with associated Langlands parameters $\alpha_{_{\mathcal P,\Phi}}(s)$ as defined in (\ref{langlandsparamsforPhi}).  Let $M = (m_1,m_2,\ldots, m_{n-1}) \in \mathbb Z_{>0}^{n-1}$.  Then the $M^{th}$ term in the Fourier-Whittaker expansion of $E_{\mathcal P, \Phi}$ is
\begin{align*}
\int\limits_{U_n(\mathbb Z)\backslash U_n(\mathbb R)} E_{\mathcal P, \Phi}(ug, s)\, \overline{\psi_M(u) }\; du \; = \; \frac{A_{E_{\mathcal P, \Phi}}(M,s)}
{\prod\limits_{k=1}^{n-1}  m_k^{k(n-k)/2}} \; W_{\alpha_{_{\mathcal P,\Phi}}(s)}\big(M g\big),
\end{align*}
where $A_{E_{\mathcal P, \Phi}}(M,s) = A_{E_{\mathcal P, \Phi}}\big((1,\ldots,1),s\big) \cdot \lambda_{E_{\mathcal P, \Phi}}(M,s),$
and
 \begin{align}\label{HeckeCoeff}
    \lambda_{E_{P,\Phi}}\big((m,1,\ldots,1),s\big) & = \hskip-10pt\underset {c_1c_2\cdots c_r = m } {\sum_{    c_1, c_2, \ldots, c_r \, \in \, \mathbb Z_{>0}}} \hskip-10pt \lambda_{\phi_1}(c_1)  \cdots \lambda_{\phi_r}(c_r)
    \cdot c_1^{s_1}  \cdots c_r^{s_r}.
\end{align}
 is the $(m,1,\ldots,1)^{th}$ (or more informally  the $m^{th}$) Hecke eigenvalue of $E_{P,\Phi}$.
 \end{proposition}
 
\begin{proof}
The proof of (\ref{HeckeCoeff}) is given in \cite{Goldfeld2006}.
\end{proof}

 \begin{theorem}[The first Fourier coefficient of   $E_{\mathcal P, \Phi}$] \label{Thrm:FirstEisCoeff}
 Assume that each Maass form $\phi_k$ (with $1\le k\le r$) occurring in $\Phi$ has Langlands parameters
  $\alpha^{(k)} := (\alpha_{k,1},\ldots,\alpha_{k,n_k})$ with the convention that if $n_k=1$ then $\alpha_{k,1}=0.$ We also assume that each $\phi_k$ is normalized to have Petersson norm $\langle \phi_k, \phi_k\rangle = 1.$
 Then the first coefficient of $E_{\mathcal P, \Phi}$ is given by
 \begin{align*} 
 & A_{E_{\mathcal P, \Phi}}\big((1,\ldots,1),s\big) =  \underset{n_k\ne1}{\prod_{k=1}^r} L^*\big(1, \text{\rm Ad}\; \phi_k\big)^{-\frac12 }\hskip-4pt
  \prod_{1\le j<\ell\le r}   L^*\big(1+s_j-s_{\ell}, \;\phi_j\times\phi_{\ell}\big)^{-1}
 \end{align*} 
 up to a non-zero constant factor with absolute value depending  only on $n$. 
Here
$$L^*(1,\,\text{\rm Ad}\;\phi_k) = L(1,\,\text{\rm Ad}\;\phi_k) \prod_{1\le i \ne j\le n_k}  \Gamma\left(\frac{1+\alpha_{k,i}-\alpha_{k,j}}{2}   \right)$$
and
$$L^*(1+s_j-s_\ell, \;\phi_j\times\phi_\ell) = \begin{cases} L^*(1+s_j-s_\ell,\, \phi_j) & \text{if}\; n_\ell =1 \;\text{and}\;n_j\ne 1,\\
L^*(1+s_j-s_\ell,\, \phi_\ell) & \text{if}\; n_j=1 \;\text{and}\;n_\ell\ne 1,\\
\zeta^*(1+s_j-s_\ell) & \text{if}\; n_j=n_\ell=1. 
\end{cases}
$$
Otherwise, $L^*(1+s_j-s_\ell, \phi_j\times\phi_\ell)$ is the completed Rankin-Selberg L-function.
\end{theorem}

\begin{proof}
 To prove Theorem \ref{Thrm:FirstEisCoeff} we apply the template method introduced in \cite{GMW2021}. In the template protocol we replace each cusp form  $\phi_k$ in $\Phi$ with a (smaller) Borel Eisenstein series 
$$E_\mathcal B\left(*, \,\alpha^{(k)}\right)$$
with the same Langlands parameters as $\phi_k.$
The next step is to determine the correct normalization of $E_\mathcal B\left(*, \,\alpha^{(k)}\right)$. Since $\phi_k$ has Petersson norm =1, it follows from \cite{GSW21} that the first Fourier coefficient of $\phi_k$ (denoted $A_{\phi_k}(1,\ldots,1)$) is given by
 $$A_{\phi_k}(1,\ldots,1) = \begin{cases} L(1, \text{\rm Ad}\,\phi_k)^{-\frac12} \prod\limits_{1\le i<j\le n_k} \Gamma\Bigl(\frac{1+\alpha_{k,i}-\alpha_{k,j}}{2}   \Bigr)^{-1}& \text{if} \; n_k > 1,\\
 \; 1 & \text{if} \; n_k = 1.
 \end{cases}
 $$
  up to a non-zero constant factor with absolute value depending  only on $n$. This together with  (\ref{FirstCoeffPMin}) shows that 
  \begin{equation}\label{replacement}
   A_{\phi_k}(1,\ldots,1) \left(\,\prod_{1\le i<j\le n_k} \zeta^*\left(1+\alpha_{k,i}-\alpha_{k,j}   \right)\right) \cdot E_\mathcal B\left(*, \,\alpha^{(k)}\right)
  \end{equation}
   has exactly the same first coefficient as $\phi_k$ up to a non-zero constant factor with absolute value depending  only on $n$.

\vskip 5pt
Recall the Langlands parameters  of $E_{\mathcal P, \Phi}(g,s)$ (denoted $\alpha_{_{\mathcal P,\Phi}}(s)$)  given by
\begin{align}\label{langlandsparamsforPhi2}
&\bigg (\overbrace{\alpha_{1,1}+s_1, \;\ldots \;,\alpha_{1,n_1}+s_1}^{n_1 \;\,\text{\rm terms}}, \quad\overbrace{\alpha_{2,1}+s_2, \;\ldots \;,\alpha_{2,n_2}+s_2}^{n_2 \;\,\text{\rm terms}},\\\nn&\hskip160pt
\quad\ldots \quad
,\overbrace{\alpha_{r,1}+s_r, \;\ldots \;,\alpha_{r,n_r}+s_r}^{n_r \;\,\text{\rm terms}}\bigg).  \end{align}
\vskip 5pt
 By replacing each $\phi_k$ with (\ref{replacement}) we may form a new Borel Eisenstein series $E_{\mathcal B, \text{\rm new}}$ with Langlands parameters given by (\ref{langlandsparamsforPhi2}).
We then apply Proposition \ref{FirstCoeffPMin} to obtain the first coefficient of $E_{\mathcal B, \text{\rm new}}$ which takes the form

\begin{align*}
&\Biggl[\;\underset{n_k\ne 1}{\prod_{1\le k \le r}}L(1, \text{\rm Ad}\,\phi_k)^{-\frac12} \prod\limits_{1\le i<j\le n_k} \Gamma\biggl(\frac{1+\alpha_{k,i}-\alpha_{k,j}}{2}   \biggr)^{-1}\\&\quad\ \cdot\biggl(\prod_{1\le i<j\le n_k} \zeta^*\big(1+\alpha_{k,i}-\alpha_{k,j}   \big)\biggr)\Biggr]
\left(\prod_{1\le k\le r}\prod_{1\le i<j\le n_k} \zeta^*\big(1+\alpha_{k,i}-\alpha_{k,j}    \big)\right)^{-1}\\
&\quad\ 
\cdot 
\Biggl(\prod_{1\le k <\ell\le r} \;\prod_{1\le i\le n_k} \prod_{1\le j\le n_\ell} \zeta^*\big(1+s_k-s_\ell+\alpha_{k,i} -\alpha_{\ell,j}  \big)\Biggr)^{-1}\\& = \Biggl(\underset{n_k\ne 1}{\prod_{1\le k \le r}} L^*(1, \,\text{\rm Ad} \,\phi_k)^{-\frac12}\Biggr)\\&\quad\ \cdot \Biggl(\prod_{1\le k <\ell\le r}  \;\prod_{1\le i\le n_k} \prod_{1\le j\le n_\ell} \zeta^*\big(1+s_k-s_\ell+\alpha_{k,i} -\alpha_{\ell,j}  \big)\Biggr)^{-1} ,
\end{align*}
up to a non-zero constant factor with absolute value depending  only on $n$.

\vskip 5pt
By the template method, the occurrence of $\zeta^*\big(1+s_k-s_\ell+\alpha_{k,i} -\alpha_{\ell,j}  \big)$ in the first coefficient of $E_{\mathcal B, \text{\rm new}}$ tells us that  $L^*\big(  1+s_k-s_{\ell}, \;\phi_k\times\phi_{\ell}\big)$ is the corresponding component of the first coefficient of  $E_{\mathcal P, \Phi}(g,s)$ provided neither $\phi_k$ or $\phi_\ell$ are the constant function one. The other cases (when one or both of $\phi_k, \,\phi_\ell$ equal 1)  follow in a similar manner.
\end{proof}

\subsection*{Acknowledgements}
 Dorian Goldfeld is partially supported by Simons Collaboration Grant Number 567168.

\end{document}